\documentclass[11pt]{article}
\usepackage{amssymb}
\newtheorem{precor}{{\bf Corollary}}

\newenvironment{cor}{\begin{precor}{\hspace{-0.5
               em}{\bf.\ }}}{\end{precor}}
\newtheorem{precon}{{\bf Conjecture}}

\newtheorem{prealphcon}{{\bf Conjecture}}

\newtheorem{predefin}{{\bf Definition}}

\newtheorem{preexm}{{\bf Example}}

\newtheorem{preappl}{{\bf Application}}

\newtheorem{prelem}{{\bf Lemma}}

\newtheorem{preproof}{{\bf Proof.\ }}

\newenvironment{proof}[1]{\begin{preproof}{\rm
               #1}\hfill{$\blacksquare$}}{\end{preproof}}
\newtheorem{pretheorem}{{\bf Theorem}}

\newenvironment{theorem}{\begin{pretheorem}{\hspace{-0.5
               em}{\bf.\ }}}{\end{pretheorem}}
\newtheorem{prealphtheorem}{{\bf Theorem}}

\newtheorem{prealphlem}{{\bf Lemma}}

\newtheorem{prepro}{{\bf Proposition}}

\newtheorem{preprb}{{\bf Problem}}

\newtheorem{prerem}{{\bf Remark}}

\newtheorem{preapp}{{\bf Application}}

\newtheorem{prequ}{{\bf Question}}

\def\conct[#1,#2]{\mbox {${#1} \leftrightarrow {#2}$}}
\def\dconct[#1,#2]{\mbox {${#1} \rightarrow {#2}$}}
\def\deg[#1,#2]{\mbox {$d_{_{#1}}(#2)$}}
\def\mindeg[#1]{\mbox {$\delta_{_{#1}}$}}
\def\maxdeg[#1]{\mbox {$\Delta_{_{#1}}$}}
\def\outdeg[#1,#2]{\mbox {$d_{_{#1}}^{^+}(#2)$}}
\def\minoutdeg[#1]{\mbox {$\delta_{_{#1}}^{^+}$}}
\def\maxoutdeg[#1]{\mbox {$\Delta_{_{#1}}^{^+}$}}
\def\indeg[#1,#2]{\mbox {$d_{_{#1}}^{^-}(#2)$}}
\def\minindeg[#1]{\mbox {$\delta_{_{#1}}^{^-}$}}
\def\maxindeg[#1]{\mbox {$\Delta_{_{#1}}^{^-}$}}

\def\dre[#1,#2,#3]{\mbox {${\cal E}^{^{#3}}(#1,#2)$}}
\def\var[#1,#2]{\mbox {${\rm Var}_{_{#1}}(#2)$}}
\def\ls[#1]{\mbox {$\xi^{^{#1}}$}}
\def\hom[#1,#2]{\mbox {${\rm Hom}({#1},{#2})$}}
\def\onvhom[#1,#2]{\mbox {${\rm Hom^{v}}(#1,#2)$}}
\def\onehom[#1,#2]{\mbox {${\rm Hom^{e}}(#1,#2)$}}
\def\core[#1]{\mbox {$#1^{^{\bullet}}$}}
\def\cay[#1,#2]{\mbox {${\rm Cay}({#1},{#2})$}}
\def\sch[#1,#2,#3]{\mbox {${\rm Sch}({#1},{#2},{#3})$}}
\def\cays[#1,#2]{\mbox {${\rm Cay_{s}}({#1},{#2})$}}
\def\dirc[#1]{\mbox {$\stackrel{\rightarrow}{C}_{_{#1}}$}}
\def\cycl[#1]{\mbox {${\bf Z}_{_{#1}}$}}

\begin{document}

\begin{center} 
{\Large \bf A note on fall colorings of Kneser graphs}\\
\vspace{0.3 cm}
{\bf Saeed Shaebani}\\
{\it School of Mathematics and Computer Science}\\
{\it Damghan University}\\
{\it P.O. Box {\rm 36716-41167}, Damghan, Iran}\\
{\tt shaebani@du.ac.ir}\\ \ \\
\end{center}
\begin{abstract}
\noindent A fall coloring of a graph $G$ is a proper 
coloring of $G$ with $k$ colors such that each vertex sees all $k$
colors on its closed neighborhood. In this short note, we characterize
all fall colorings of Kneser graphs of type $KG(n,2)$.
\\


\end{abstract}
\section{Introduction}

In this note, simple graphs whose vertex sets are
nonempty and finite are considered. Also, for each positive integer $k$, the
symbol $[k]$ means $\{i|\ i\in \mathbb{N},\ 1\leq
i\leq k \}$.

Let $G=(V(G),E(G))$ be a graph. A
{\it coloring} of $G$ is a function $f:V(G)\rightarrow C$ such
that for each $c$ in $C$, the set $f^{-1}(c)$ is independent; in
this case, we consider each $c$ in $C$ as a {\it color} and call
$f^{-1}(c)$ a {\it color class} of $f$. The {\it chromatic number} 
of $G$, denoted by $\chi(G)$, is the minimum 
cardinality of a set $C$ that a coloring $f:V(G)\rightarrow C$ exists.

Let $G$ be a graph and $f:V(G)\rightarrow C$ be a coloring of $G$. 
The vertex $v$ of $G$ is called a {\it b-dominating} vertex with
respect to $f$ if $f(N[v])=C$, i.e., the vertex $v$ sees all 
colors on its closed neighborhood. The coloring $f$
is said to be a {\it fall coloring} of $G$ if all vertices of $G$ are
b-dominating \cite{dun}. We mean by ${\rm Fall}(G)$ the set of all
natural numbers $k$ for which $G$ admits a fall coloring $f:V(G)\rightarrow C$
with $|C|=k$. We may have ${\rm Fall}(G)=\emptyset$; for
example, ${\rm Fall}(C_{5})=\emptyset$, where $C_{5}$ is 
the cycle with five vertices. A graph $G$ is called {\it $f$-continuous}
if either ${\rm Fall}(G)$ is empty or ${\rm Fall}(G)\neq\emptyset$ and 
${\rm Fall}(G)=\{ k\in \mathbb{N} |\  \min({\rm Fall}(G)) \leq k \leq \max({\rm Fall}(G))\}$. 
Not all graphs are $f$-continuous;
for example, the three dimensional cube $Q_{3}$ satisfies ${\rm Fall}(Q_{3})=\{ 2,4\}$.

Let $G$ be a graph. A {\it b-coloring} of $G$ is a coloring $f:V(G)\rightarrow C$ 
such that each color class contains at least one b-dominating vertex \cite{irv}.
The set of all natural numbers $k$ that $G$ has a b-coloring $f:V(G)\rightarrow C$
with $|C|=k$, is denoted by ${\rm B}(G)$. Always $\chi(G) \in {\rm B}(G)$.
The maximum of the set ${\rm B}(G)$, say $\chi _{b} (G)$, is called the {\it b-chromatic number} of $G$.
A {\it b-continuous} graph stands for a graph $G$ with ${\rm B}(G)=\{ k\in \mathbb{N} |\ 
 \chi (G) \leq k \leq \chi _{b} (G)\}$. The three dimensional cube $Q_{3}$ is not b-continuous;
since ${\rm B}(Q_{3})=\{ 2,4\}$.

\section{Fall colorings of Kneser graphs}

Suppose that $n\geq m$. Hereafter, ${[n] \choose m}$ denotes
the set of all $m$-subsets of $[n]$. The {\it Kneser graph}
$KG(n,m)$ has the vertex set ${[n] \choose m}$, in which $A \sim
B$ iff $A \cap B = \emptyset$. In \cite{jav}, Javadi and Omoomi determined the
b-chromatic number of all Kneser graphs of type $KG(n,2)$. Also, they proved that 
$KG(n,2)$ is b-continuous whenever $n\geq 17$. In this short note, we study 
fall colorings of Kneser graphs of type $KG(n,2)$. We show that if $n\geq 2$, then
$|{\rm Fall}(KG(n,2))| \leq 1$; and as a corollary, these graphs are $f$-continuous.
The procedure uses some facts mentioned in \cite{jav}; nevertheless, it is self-contained.

\begin{theorem}{ For each natural number
$n\geq2$, we have
\\
\\
${\rm Fall}(KG(n,2))=\left\{\begin{array}{ll}
  \{1\} & n=2\ {\rm or}\ 3 \\
  \{2\} & n=4\\
  \{\frac{n(n-1)}{6}\}& n\geq5,\ n=1\ {\rm or}\ 3\ ({\rm mod}\ 6) \\
  \{\frac{(n-1)(n-2)}{6}+1\} & n\geq5,\ n=2\ {\rm or}\ 4\ ({\rm mod}\ 6) \\
  \emptyset & n\geq5,\ n=0\ {\rm or}\ 5\ ({\rm mod}\ 6)
\end{array}\right.$
}
\end{theorem}

\begin{proof}{
Obviously, ${\rm Fall}(KG(2,2))={\rm Fall}(KG(3,2))=\{1\}$. Also,
since the edge set of $KG(4,2)$ is a matching, therefore, ${\rm Fall}(KG(4,2))=\{2\}$.
 Now, let us suppose that $n\geq5$. We note that by considering the complete graph 
 $K_{n}$ with vertex set $[n]$, the Kneser graph $KG(n,2)$ is exactly the complement graph
of the line graph of $K_{n}$. Therefore, one can think of
$KG(n,2)$ as the graph whose vertex set is $E(K_{n})$ in which
two elements of $E(K_{n})$ are adjacent in $KG(n,2)$ iff they
have not any common vertices in $K_{n}$. Now, suppose that
$f$ is a fall coloring of $KG(n,2)$ and $S$
is an arbitrary color class of $f$. We have the following two cases:

\noindent Case I : The case that $\bigcap_{A\in S}A =\emptyset$. In this case, there
exist three pairwise distinct elements $a, b, c \in [n]$ such that $ S =\{ \{a,b\},\{b,c\},\{c,a\} \}$. So,
$S$ is the set of edges of a triangle in $K_{n}$. We call such color classes {\it triangular}.

\noindent Case II : The case that $\bigcap_{A\in S}A\neq\emptyset$. Let
$i\in \bigcap_{A\in S}A$. Each vertex $B$ of
$KG(n,2)$ which contains $i$ is an element of $S$; otherwise,
$B\notin S$ and $B$ has not any neighbors in $S$, contradicting
the fact that $f$ is a fall coloring. Also, since
$n\geq 5$, for each $j\in [n]\setminus \{i\}$, there exists a
vertex $C$ of $KG(n,2)$ with $i\in C$ and $j\notin C$ and $C\in
S$; therefore, $\bigcap_{A\in S}A=\{i\}$ and $S=\{A|\ A\in {[n]
\choose 2},\  i\in A\}$. Accordingly, $S$ is the set of 
all edges in $K_{n}$ that are incident with $i$. We call such color classes {\it starlike}. It
is obvious that the number of starlike color classes of $f$ is at most one.

We conclude that the set of color classes of a fall coloring of $KG(n,2)$ is either a partition of
$E(K_{n})$ into one starlike and some triangles or a
partition of $E(K_{n})$ into triangles. Conversely, every such a partition of $E(K_{n})$ is the 
set of color classes of a fall coloring of $KG(n,2)$. On the other hand, it is well-known that the
edge set of a complete graph can be partitioned into triangles iff the number of its vertices is
congruent to $1$ or  $3$ modulo $6$.
Therefore, if ${\rm Fall}(KG(n,2))\neq\emptyset$, then $n=1\
{\rm or}\ 2\ {\rm or}\ 3\ {\rm or}\ 4\ ({\rm mod}\ 6)$. So, if 
$n=0\ {\rm or}\ 5\ ({\rm mod}\ 6)$, then ${\rm Fall}(KG(n,2))=\emptyset$.
Also, for $n=1\ {\rm or}\ 3\ ({\rm mod}\ 6)$, we have 
${\rm Fall}(KG(n,2))=\{\frac{n(n-1)}{6}\}$.
Finally, if $n=2\ {\rm or}\ 4\ ({\rm mod}\ 6)$,
we can make a partition of $E(K_{n})$ into one
starlike and some triangles; and therefore, ${\rm Fall}(KG(n,2))=\{\frac{(n-1)(n-2)}{6}+1\}$.
}
\end{proof}

The following corollary is an immediate consequence of the theorem.
\begin{cor}{ For $n\geq 2$, the Kneser graph $KG(n,2)$ is $f$-continuous.
}
\end{cor}

\end{document}